\documentclass[12pt]{article}
\usepackage[english]{babel}
\usepackage{amssymb, amsmath,amsthm, verbatim, cancel, fancybox, graphicx, animate, enumerate, enumitem, amstext, array, cleveref, colonequals, ifthen, booktabs, float, afterpage, caption, subcaption, multirow, tikz-cd, blindtext, setspace, layout, url}
\usepackage[top=1in, bottom=1in, left=1.5in, right=1in]{geometry}

\newtheorem{thm}{Theorem}[section]
\newtheorem{lem}[thm]{Lemma}

\newtheorem{cor}[thm]{Corollary}

\theoremstyle{remark}
\newtheorem{rmk}[thm]{Remark}
\theoremstyle{definition}
\newtheorem{defn}[thm]{Definition}

\newcommand{\ds}{\displaystyle}
\newcommand{\x}{\times}
\let\temp\phi
\let\phi\varphi
\let\varphi\temp

\newcommand{\urlb}[1]{\textcolor{blue}{\url{#1}}}

\newcommand{\strutdepth}{\dp\strutbox}
\newcommand{\marginalnote}[1]
   {\strut\vadjust{\kern-\strutdepth\domarginalnote{#1}}}
\newcommand{\domarginalnote}[1]{\vtop to \strutdepth{
  \baselineskip\strutdepth
   \vss\llap{ #1\ \ }\null}}  
\newcounter{showlabelflag}
\newcounter{makelabelflag}

\newcommand{\makelabels}{\setcounter{makelabelflag}{1}}
\newcommand{\hidelabels}{\setcounter{showlabelflag}{2}}
\newcommand{\mylabel}[1]{
  \ifthenelse{\value{makelabelflag}=1}
    {\label{#1}}{}
  \ifthenelse{\value{showlabelflag}=1}
    {\marginpar{#1}}{}\relax}

 \makelabels
\hidelabels

\title{\vspace*{0.5in}\bf Locally solvable subgroups of PLo(I) are countable}
\author{Amanda Taylor \\ Department of Mathematics and Computer Science \\ Alfred University \\ 1 Saxon Drive \\ Alfred, NY 14802 \\ tayloral@alfred.edu }
\date{}

\begin{document}

\pagenumbering{roman}
\maketitle \thispagestyle{empty}

\centerline{\bf Abstract}
We show every locally solvable subgroup of PLo(I) is countable.  A corollary is that an uncountable wreath product of copies of $\mathbb{Z}$ with itself does not embed into PLo(I).

\section{Introduction}\mylabel{intro}
\pagenumbering{arabic} 

{\bf PLo(I)} is the group of piecewise linear orientation-preserving homeomorphisms of the unit interval with finitely many points of non-differentiability ({\bf breakpoints}). Thompson's Group F is the subgroup of PLo(I) whose elements have breakpoints occurring at dyadic rationals and slopes are integer powers of 2.  Thompson's Group F has particularly interesting properties which endear it to the field of geometric group theory.  Most notably, it is known to contain no free subgroups of rank 2 or higher \cite{nf2}, yet, despite many valiant attempts, whether the group is amenable has been an open problem since at least the late '70s.  

By contrast, the subgroup structure problem has remained poorly understood, but it is becoming better understood.  In 2005, Matt Brin studied a family of elementary amenable subgroups in \cite{ea}. Furthermore, the solvable subgroups were classified completely by Collin Bleak in \cite{geo} and \cite{alg}---papers crafted from his 2006 PhD thesis.  

The current paper is a portion of A. Taylor's PhD thesis from 2017 which focuses on a generalization of solvable subgroups called locally solvable subgroups.  A group is {\bf locally solvable} if every finitely generated subgroup is solvable.  We prove the following.

\begin{thm}\mylabel{main}
Every locally solvable subgroup of PLo(I) is countable.
\end{thm}

Hall's generalized wreath product construction in Section 1.3 of \cite{hall} results in an uncountable locally solvable
group if the indexing set $\Lambda$ is uncountable and the groups $H_{\lambda}$ are all copies of $\mathbb{Z}$.  With this construction in mind, a corollary of the main theorem is

\begin{cor}\mylabel{cool}
An uncountable wreath product of copies of $\mathbb{Z}$ with itself does not embed in PLo(I).
\end{cor}

Results in $\cite{loc solv}$ permit us to use geometry to work with locally solvable subgroups.   In particular, a notion called a transition chain is of central importance to this paper.  In \cite{geo}, Bleak shows transition chains are obstructions to building solvable subgroups.  The geometric nature of PLo(I) helps Bleak classify the solvable subgroups completely.  The foundations he builds also assist Bleak in providing an algebraic classification for the solvable subgroups in \cite{alg}.  The following conventions and terminology assist us in defining a transition chain.

 Our group actions are right actions, so we use the notation $(x)fg$ to mean $((x)f)g$.  Also, for conjugation and commutators, $f^g := g^{-1}fg$ and $[f,g] := f^{-1}g^{-1}fg$.  Given a function $f \in PLo(I)$, $Supp(f) := \{x \in I \, | \, (x)f \neq x \}$ which we call the {\bf support} of $f$.  Similarly, given a subgroup $G$ in PLo(I), the {\bf support} of $G$ is $Supp(G) := \{x \in I \, | \, (x)f \neq x \text{ for some } f \in G\}$.  An {\bf orbital} of a function $f$ in $PLo(I)$ is a maximal interval contained in $Supp(f)$. Orbitals are open intervals in $\mathbb{R}$, and every function in PLo(I) has only finitely many orbitals because they have only finitely many breakpoints.  A {\bf bump} of $f$ is the restriction of $f$ to a single orbital.  A function is called a {\bf one-bump function} if it has only one orbital. 

A {\bf signed orbital} is a pair $(A,f)$ where $f \in PLo(I)$ and $A$ is an orbital of $f$.   We call $A$ the {\bf orbital} of $(A,f)$ and $f$ the {\bf signature}.  Signed orbitals are symbols used to represent bumps of functions, and they are clearly in one-to-one correspondence with bumps of functions.  Thus we move freely between the two concepts.  We utilize terminology for sets related to signed orbitals, so we define those now.

For a set $C$ of elements of $PLo(I)$, let $\mathcal{SO}(C)$ be the set of all signed orbitals of $C$ and $\mathcal{O}(C)$ the set of all the orbitals of elements in 
$C$.  For a set with a single element $g \in$ PLo(I), we will omit set brackets in the previous notations and write only $\mathcal{SO}(g)$ or $\mathcal{O}(g)$. If $C$ is a collection of signed orbitals of PLo(I), we also use the notations $\mathcal{O}(C)$ and $\mathcal{S}(C)$ to denote the set of all orbitals of members of $C$ and the set of all signatures of members of $C$, respectively.

\begin{defn}
 A {\bf transition chain} is a pair of signed orbitals $(A,f)$, $(B,g)$ such that $A \cap B \neq \emptyset$, $A \not\subset B$, and $B \not\subset A$.  We say a subset $C$ of signed orbitals of PLo(I) has {\bf no transition chains} if no pair of elements of $\mathcal{SO}(C)$ is a transition chain.  
\end{defn}

Visually, a transition chain is a pair of bumps which overlap, but neither of the corresponding orbitals properly contains the other.  We may also say that two functions $f$ and $g$ form a transition chain if the set $\mathcal{SO}(\{f,g\})$ contains a pair forming a transition chain.  See figure \ref{tchain} for an example of functions who graphs form a transition chain. 

   \begin{figure}
   \begin{center}
     \includegraphics[width=0.45\linewidth]{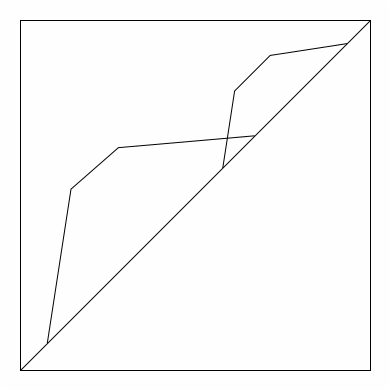}
     \caption{Graphs of elements with signed orbitals that form a transition chain}\mylabel{tchain}
        \end{center}
   \end{figure}
   
 When transition chains are not allowed, group actions on the interval are much simpler.  Furthermore, thanks to Theorem 3.2 in \cite{loc solv}, we can use lack of transition chains to work geometrically with locally solvable subgroups of PLo(I):


\begin{thm}\mylabel{loc solv}
A subgroup of $PLo(I)$ is locally solvable if and only if it has no transition chains.
\end{thm}

 Henceforth, we work with the geometric notion of a group without transition chains instead of the algebraic notion of locally solvable subgroups.  The structure of the rest of the paper is 1. Several more definitions and elementary facts, 2. Prove that $\mathcal{SO}(G)$ is countable if $G$ is a group which has no transition chains, and 3. Consider some consequences of the proof of the main theorem.

\subsection{Additional Definitions and Facts}

One definition which plays an important role is a fundamental domain.  A {\bf fundamental domain} of a function $f$ is a half open interval $[x,xf)$ where $x$ is in the support of $f$.  A {\bf fundamental domain} of a signed orbital $(A,f)$ is a half open interval $[x,xf)$ where $x$ is a point in the orbital $A$. A set $S$ of signed orbitals is {\bf fundamental} if whenever $(A, f), (B, g) \in S$ are such that $A \subset B$, then $A$ is in a fundamental domain of $(B, g)$.  

We also require terminology related to partially ordered sets.  Any collection of orbitals is partially ordered by inclusion. Thus $\mathcal{O}(G)$ is a poset, and we will use some poset notation throughout this paper.  Let $(P, \leq)$ be a partially ordered set and let $A \in P$.  We write $\downarrow A = \{B \, | \, B < A \}$ and refer to this as the {\bf downset of $A$ in $P$}. Similarly, the {\bf upset of $A$ in $P$} is $\uparrow A = \{B | A < B \}$.  Note these sets do not contain $A$.  Also, there is no reference to $P$ in the notation $\downarrow A$ or $\uparrow A$, so the poset will be specified or clear from the context. 

The set $\mathcal{SO}(G)$ inherits an ordering from $\mathcal{O}(G)$ by making new chains for distinct elements with the same orbitals.  More explicitly, order $\mathcal{S}(G)$ by the trivial partial order on $G$, that is, if $a \neq b$ in $G$, $a$ and $b$ are incomparable.  Let $P_2$ be this partial order and $P_1$ be the partial order on $\mathcal{O}(G)$. Define the {\bf lexicographical partial order} on a product $A \times B$ of partially ordered sets by  $(a,b) \leq (c,d)$   if and only if $a < c$ or $(a=c \text{ and } b \leq d)$. Extend $P_1$ to a partial order on $\mathcal{SO}(G)$ by taking the product $P = P_1 \times P_2$ with the lexicographical partial ordering.  We use this partial order on $\mathcal{SO}(G)$ throughout this paper.

We call a chain in the poset $\mathcal{O}(G)$ a {\bf stack} in $G$.  A subset of a stack is also a stack, and hence is called a {\bf substack}.  A {\bf tower} is a chain of orbitals together with an assignment of a signature to each orbital. Hence a tower is set of signed orbitals which is naturally order isomorphic to its underlying stack.  A tower can also be described as a chain in $\mathcal{SO}(G)$ when equipped with partial order $P$ in the last paragraph.  Every subset of a tower is also a tower, so is called a {\bf subtower}.  We often work with totally ordered sets and refer to them simply as {\bf ordered sets}. Whenever a different type of ordering is used, it will be specified or clear from the context.  

Geometrically, one can think of a tower as bumps with nested orbitals.  However, it is possible that signatures of a tower have multiple bumps which create more complicated dynamics.  Towers were introduced in Collin Bleak's Ph.D. thesis to help classify solvable subgroups of PLo(I).

Definitions related to endpoints of orbitals will also assist us. If $A = (x,y)$ is an orbital of a function $f \in$ PLo(I), we call $x$ and $y$ the {\bf ends} of $A$.  We say something happens {\bf near an end} or {\bf near the ends} of $A$ if it true on some interval $(x,a) \subseteq A$ or true on 2 intervals $(x,a), (b,y) \subseteq A$, respectively.  Given two element orbitals $A$ and $B$, we say $A$ {\bf shares an end} with $B$ or $A$ and $B$ {\bf share an end} if $A \subset B$ or $B \subset A$ and at least one of their endpoints is the same.  If $A$ is an orbital of $f$ and $x \in A$, continuity implies that if $xf > x$ for some $x \in A$ then $xf > x$ for all $x \in A$.  In this case, we say $f$ {\bf moves points right on $A$}, and similarly for left. 

 {\bf Affine components of f} are the components of $[0,1]-B_f$ where $B_f$ is the set of breakpoints of $f$.  Affine components are naturally ordered from left to right.  The slopes of the first and last affine components of $f$ are called the {\bf initial and terminal slopes} of $f$.  Sometimes we consider a relative version of these definitions on a particular orbital $A$ of $f$, in which case we append {\bf on $A$} to any of the previous descriptions. 

The facts in the rest of this section are elementary and are left to the reader.

\begin{rmk} \mylabel{containment}
If $G$ has no transition chains, and $A,B \in \mathcal{O}(G)$ are such that $A \cap B \neq \emptyset$, then
$A \subset B$, $B \subset A$, or $A=B$.
\end{rmk}

\begin{rmk} Let $T$ be a tower in a subgroup $G$ of PLo(I).
 \begin{enumerate}[label={(\arabic*)},ref={\thecor~(\arabic*)}]
 \item Each subset of $T$ is a tower.
 \item $\mathcal{O}(T)$ and each subset of it is a stack.
 \item $\uparrow \mathcal{O}(T)$ in $\mathcal{O}(G)$ is a stack.
 \item $\downarrow \mathcal{O}(T)$ in $\mathcal{O}(G)$ may not be a stack.
 \end{enumerate}
 \end{rmk}

\begin{lem} 
 \begin{enumerate}[label={(\arabic*)},ref={\thecor~(\arabic*)}] 
 \item If $T$ is a tower of elements in PLo(I), then its underlying stack $\mathcal{O}(T)$ is order isomorphic to $T$.
\item  \mylabel{basic conj orbitals} If $g, c \in$ PLo(I) and the orbitals of $g$ are $\{A_i \, | \, 1 \leq i \leq n \}$, then the orbitals of $g^c$ are $\{A_i c \, | \, 1 \leq i \leq n\}$, and the map $\phi: \mathcal{O}(g) \longrightarrow \mathcal{O}(g^c)$ which takes $A_i$ to $A_i c$ is a bijection.
\item \mylabel{basic conj towers} If $T = \{(A_i,a_i) \, | \, i \in I \}$ is a tower in PLo(I) and $c \in$ PLo(I), then the set $T^c \colonequals \{(A_ic,a_i^c) | i \in I \}$ is a tower and the map $\phi_c : T \longrightarrow T^c$ defined by $(O,g) \longmapsto (Oc, g^c)$ is an isomorphism of ordered sets. A similar result holds for stacks.
\item \mylabel{conj slopes} If $g,c \in PLo(I)$, then $g^c$ has the same initial and terminal slopes on each of its orbitals as $g$ has on its corresponding orbitals.  
\item \mylabel{moveover}  Let $G \leq PLo(I)$, $O$ be an orbital of $G$, and $x,y \in O$ with $x < y$.  Then $\exists g \in G$ such that $xg > y$.
\end{enumerate} 
\end{lem}

\begin{lem} \mylabel{multiply basics} Let $G$ be a group without transition chains and $a,b \in G$ with orbitals $O_a, O_b$, respectively.
 \begin{enumerate}[label={(\arabic*)},ref={\thecor~(\arabic*)}]
  \item \mylabel{end}  If $O_a$ and $O_b$ share an end, then $O_a = O_b$.  The contrapositive is also very useful: If $O_a \neq O_b$, then $a$ and $b$ do not share an end.
 \item \mylabel{fd} If $O_a$ is properly contained in $O_b$, then $O_a$ is in a fundamental domain of $(O_b,b)$.  Thus towers in a group without transition chains are always fundamental.
\item \mylabel{multiply} If $O_{a_1} \subset O_{a_2} \subset \cdots \subset O_{a_n}$ is a proper chain of orbitals of $a_1, a_2, \cdots, a_n$, respectively, then $O_{a_n}$ is an orbital of the products $a_1a_2 \cdots a_n$ and $a_na_{n-1}\cdots a_1$.
 \end{enumerate}
 \end{lem}

\section{Countability}\mylabel{sectioncountable}

In this section, we prove the following

\begin{thm}\mylabel{countable}
A subgroup of PLo(I) without transition chains is countable.
\end{thm}

 The theorem follows from the next 3 lemmas which result in a proof that $\mathcal{SO}(G)$ is countable if $G$ is a group with transition chains.

For the reminder of the section, let $G$ be a subgroup of PLo(I) which has no transition chains.

\begin{lem}\mylabel{towers count}
Towers and stacks in $G$ are countable.
\end{lem}

\begin{proof}
Since every stack is in bijection with some tower by simply picking signatures for each of the orbitals, it is enough to show that the underlying stack of every tower is countable. Let $T$ be a tower in $G$ and $(A,f) \in T$.  Let $A=(a,b)$.  Elements of $T$ are in bijection with elements of $\mathcal{O}(T)$. We produce a collection of disjoint open intervals of $[0,1]$ which are in bijection with elements of $\mathcal{O}(T)$.  Since each of these intervals contains a rational, the set of all the intervals is countable.
Consider the downset of $A$ in $\mathcal{O}(T)$.  We claim there is a $c$ with $a< c < b$ such that $(a,c) \cap C  = \emptyset$ for every $C \in \downarrow A$. The interval $(a,c)$ is the one we seek.  The claimed property of $c$ will imply all these 
intervals are disjoint.

If $f$ moves points left on $A$, then replace $f$ with its inverse.  This does not hinder our argument, since $f$ and $f^{-1}$ have the same orbitals.  Let $(B,g)$ be an element of $\downarrow (A,f)$ in T.  Then $B$ separates $\downarrow A \text{ in } \mathcal{O}(T)$ into two pieces: $\uparrow B$ in $\downarrow A$ and $\downarrow B$ in $\downarrow A$.  Let $x \in B$ and $c = (x)f^{-1}$.  Then $(a,c) \cap B = \emptyset$ due to Lemma \ref{fd}.  Also $(a,c)\cap C = \emptyset$ for any $C \in \downarrow B$. Furthermore, since intervals in $\uparrow B$ contain $x$ and are contained in a fundamental domain of $(A,f)$, they do not contain $(x)f^{-1}$.  Thus intervals in $\uparrow B$ have left endpoints larger than $c$ and so $C \cap (a, c) = \emptyset$ for any $C \in \uparrow B$.  Therefore, $C \cap (a,c) = \emptyset$ for all $C \in \downarrow A$ and the proof is complete. 
\end{proof}

\begin{lem}\mylabel{orbitals count}
The set $\mathcal{O}(G)$ is countable.
\end{lem}
\begin{proof}
Let $L$ be the set of all lengths of elements in $\mathcal{O} (G)$.  $L$ is some subset of $(0,1]$. 
Let
$m : \mathcal{O} (G) \rightarrow L$ be the usual measure on intervals, so $m$ maps each orbital to its length.  

For each positive integer $n$, let $I_n = \left(\left( 2/3 \right)^n, \left(2/3 \right)^{n-1}\right]$.  The set $\displaystyle C = \{I_n | n \in \mathbb{N} \}$ is a partition of the interval $(0,1]$.  Since C is countable, it is enough to show that $m$ maps countably many elements of $\mathcal{O}(G)$ into each element of $C$.  Consider an arbitrary element $I_n$ of $C$, let $R$ be the set of all elements of $\mathcal{O}(G)$ which $m$ maps into $I_n$, and let $K$ be the union of elements of $R$. 

Equip $\mathbb{R}$ with the topology generated by open intervals, and $K$ with the corresponding subspace topology.  Since $\mathbb{R}$ has a countable basis, so does $K$.  Therefore, the open cover $R$ of $K$ has a countable subcover $S$.  If every element of $S$ intersected only countably many elements of $R$, then $R$ would be countable.  Hence, we will show every element of $S$ intersects only countably many elements of $R$.  Let $O$ be an arbitrary element of $S$.
Let $U$ be the subset of $R$ whose elements intersect $O$.   Our aim now is to show $U$ is countable.  We do this by showing $U$ is a stack and therefore countable by the previous lemma.

To show $U$ is a stack, we must show that every pair of distinct elements $A,B \in U$ are comparable.  Since $G$ has no transition chains, intersection of orbitals implies containment.  Thus we can divide $U - \{O\}$ into two pieces:  $\uparrow O$ and $\downarrow O$, 
or those properly containing $O$ and those properly contained in $O$, respectively. Assume toward a contradiction that $A$ and $B$ are disjoint.  Then they must be contained in $\downarrow O$.  Because they are contained in $O$, one of them will have length less one-half the length of $O$ by Lemma \ref{end}. Assume it is $A$. Since $m(A), m(O) \in I_n$, multiplying their lengths by 2/3 results in a number which is in $I_{n-1}$, hence not in $I_n$.  In particular, $\ds \frac{2}{3} m(O) \leq m(A)$.  Thus $\ds \frac{2}{3} m(O) \leq m(A) \leq \frac{1}{2} m(O)$, a contradiction because $m(A) > 0$.  Therefore $A$ and $B$ are not disjoint, so they must be comparable.
\end{proof}

The next lemma will require the following definition.

Define a {\bf bouncepoint} of a pair $f,g$ of PL functions to be a point $b$ where 1. $(b)f = (b)g$, 2. there is some open interval $(b,c)$ on which $(x)g \neq (x)f$, and 3. $b$ is a breakpoint of $f$ or $g$.  By a {\bf bouncepoint} $b$ of a single function $f$, we mean $b$ is a breakpoint of $f$ and there exists some function $g$ such that $b$ is a bouncepoint of the pair $f,g$.  


 We will also need the following Lemma which is a consequence of results from Section 3.3.2 of  paper \cite{alg}:  
 
\begin{lem}\mylabel{it counts}
Given an orbital $O$ of a group without transition chains, there are at most countably many possible initial and terminal slopes for elements with that orbital. 
\end{lem}

At last, we state the final lemma for our proof of Theorem \ref{countable}.

\begin{lem}\mylabel{signatures count}
Let $F(O)$ be the set of all bumps of functions of $G$ which have orbital $O$.  Then $F(O)$ is countable for any $O \in \mathcal{O}(G)$.
\end{lem}

\begin{proof}
Let $B_n = \{f \in F(O) \, | \, f \text{ has exactly } n \text{ breakpoints}\}$.  Since $\cup^{\infty}_{i=1} B_i = F(O)$, it is enough to show that $B_n$ is countable for $n \in \mathbb{N}$. We define an injective map $\phi$ from the set $B_n$ to a countable set.  Let $f \in B_n$, $x_0$ the left endpoint of $O$, and $s_0$ the initial slope of $f$ leaving $x_0$ (i.e., the initial slope of $f$ on $O$). The function $f$ has $n$ breakpoints, so it has at most $n$ bouncepoints.  Suppose $f$ has $m$ bouncepoints.  Let the bouncepoints of $f$ be $b_1, b_2, . . . ,b_m$, and assume the order on the index set matches that of the points. Let $s_1, s_2, . . . ,s_m$ be the slopes of $f$ leaving $b_1, b_2, . . . , b_m$, respectively, i.e., $s_i$ is the slope of the affine component with left endpoint $b_i$ for $1 \leq i \leq m$. Define $\phi(f)$ to be the ordered set of information $\{s_0, b_1, s_1, b_2, s_2, . . . , b_m, s_m\}$.  

First, we argue $\phi$ is injective.  Assume $f,g \in B_n$ such that $f \neq g $.  If $f,g$ have different initial slopes, we are done, so assume they have the same initial slope.  Consider the maximal closed interval $[x_0, b]$ on which which $f = g$.  Since $f,g$ have the same initial slope, $b$ is a bouncepoint. Note that if $f,g$ don't have the same initial slope, there may not exist any bouncepoint for the pair.  Thus, it is essential that our map $\phi$ include initial slopes. Because $f(x) = g(x)$ for $x<b$, $\phi(f)$ and $\phi(g)$ are the same until the point $b$ appears in one or the other.  At that slot in the ordered sets $\phi(f), \phi(g)$, there are two possibilities: 1. $b$ is a breakpoint (hence a bouncepoint) of exactly one of $f$ or $g$; or 2. $b$ is a breakpoint of both the single functions $f$ and $g$, in which case the next slopes must differ.  In either case, $\phi(f) \neq \phi(g)$.

Now we argue Im($\phi$) is countable by showing that the choices for the $s_i$'s and the $b_i$'s are countable.  Observe that if $b$ is a bouncepoint of some pair $f,g$ in a group $G$ then $b$ is an endpoint of an orbital of $fg^{-1}$.  This is simply because $(b)f = (b)g$ and $(x)f \neq (x)g$ on some interval $(b,c)$.  Thus, by Lemma \ref{orbitals count}, the set $B$ of all possible bouncepoints of elements of $F(O)$ is countable.  Let $S$ be the set of all possible slopes leaving bouncepoints. If there were uncountably many possible slopes for $f$ emanating from some bouncepoint $b$, then applying the chain rule would result in uncountably many initial slopes emanating from $b$ for functions of the form $fg^{-1}$ where $b$ is a bouncepoint of $f,g$ and hence an orbital endpoint for $fg^{-1}$.  This is a contradiction to Lemma \ref{it counts}, so $S$ must also be countable.  Furthermore, all possible initial slopes $S_0$ are countable by Lemma \ref{it counts} (not all orbital endpoints are bouncepoints, so we must note this separately from the previous argument).  Let $S_0, S, B$ include the empty set as an element.  Then it's easy to see that Im$\phi$ injects into the ordered product $S_0 \x B \x S \x B \x S . . . \x B \x S$ where there are $n$ copies of $B \x S$. Therefore, Im($\phi$) is countable.
\end{proof}

\begin{proof}[Proof of Theorem \ref{countable}]
Define $G_n = \{g \in G \, | \, g \text{ has exactly } n \text{ orbitals}\}$.  Then, $\cup^{\infty}_{i=0} G_i = G$, so we need only show $G_n$ is countable.  $G_0$ is just the identity element, so it is countable.  An element of $G_n$ is determined by a choice of $n$ orbitals and a choice of one bump for each of those orbitals. Since each of these $2n$ choices are selected from countable sets by Lemmas \ref{orbitals count} and \ref{signatures count}, $G_n$ injects into a countable set, hence is countable. 
\end{proof} 

To state consequences of the main theorem, we start with the following definition.

Define a {\bf corner} of a pair $f,g$ of PL functions to be a point $b$ where 1. $(b)f = (b)g$, 2. there is some open interval $(b,c)$ on which $(x)f \neq (x)g$, and 3. $b$ is in the interiors of affine components for both $f$ and $g$.  

The following corollaries are related to pairs of functions and dynamics of points.

\begin{cor}
 The set of all corners of elements of $G$ is countable.
\end{cor}

\begin{proof}
 This follows from Lemma \ref{orbitals count} in the same way that bouncepoints are shown to be countable: If $c$ is a corner of the pair $f,g$ then $g^{-1}f$ has orbital beginning at $c$.  Hence every corner corresponds to some orbital of $G$, the set of which is countable.  
\end{proof}



\begin{cor}
The set of all breakpoints of elements in $G$ is countable. 
\end{cor}

\begin{proof}
Each function in $G$ is completely determined by a finite ordered list $\{b_0, b_1, b_2, \cdots, b_n\}$ of breakpoints.  Since $G$ is countable, the union of all these lists is countable.  
\end{proof}

Now we provide consequences related to subgroups of PLo(I).

\begin{cor}
Every uncountable subgroup of PLo(I) contains two elements which generate a non-solvable subgroup.
\end{cor}
\begin{proof}
Every uncountable subgroup of PLo(I) contains a transition chain by the contrapositive of Theorem \ref{countable}.  
\end{proof}

\begin{cor}
An ordered wreath product of copies of $\mathbb{Z}$ as defined by P. Hall in \cite{hall} does not embed in PLo(I) if the underlying ordered set is uncountable.
\end{cor}

\begin{proof}
Every finitely generated subgroup of an ordered wreath product of copies of $\mathbb{Z}$ is a subgroup of a finitely iterated wreath product of copies of $\mathbb{Z}$ with itself.  Therefore, every finitely generated subgroup of an ordered wreath product of copies of $\mathbb{Z}$ is solvable.  Hence an ordered wreath product of copies of $\mathbb{Z}$ is locally solvable.  By the main theorem, any such group that embeds in PLo(I) is countable.
\end{proof}

\begin{cor}
Suppose a direct sequence of locally solvable groups has uncountable direct limit.  Then the sequence does not embed as a chain in the subgroup lattice of PLo(I).  
\end{cor}

\begin{proof}
The union of a chain of locally solvable subgroups is locally solvable.  By the main theorem, such a union is countable if it is in PLo(I).
\end{proof}

\end{document}